\documentclass[reqno,english]{amsart}
\usepackage{etex}
\usepackage{amsmath,amssymb,amsthm,bbm,mathtools,comment}
\usepackage[shortlabels]{enumitem}
\usepackage[pdftex,colorlinks,backref=page,citecolor=blue]{hyperref}
\usepackage[mathscr]{euscript}
\usepackage{tikz,babel,adjustbox}
\usepackage{cmtiup}
\usepackage{amsfonts}
\usepackage{graphicx}
\usepackage{caption}
\usepackage{verbatim}
\usepackage{array}
\usepackage[frame,cmtip,arrow,matrix,line,graph,curve]{xy}
\usepackage{graphpap, color, pstricks}
\usepackage{pifont}
\usepackage{cmtiup}
\usepackage{amssymb}

\allowdisplaybreaks

\theoremstyle{plain}
\newtheorem{theorem}{Theorem}[section]

\newtheorem{lem}[theorem]{Lemma}
\newtheorem{claim}[theorem]{Claim}

\newtheorem{corollary}[theorem]{Corollary}

\newtheorem{defn}[theorem]{Definition}

\theoremstyle{remark}

\newcommand{\cal}[0]{\mathcal}

\renewcommand{\dots}[0]{\ldots}

\newcommand{\beq}[1]{\begin{equation}\label{#1}}
\newcommand{\enq}[0]{\end{equation}}

\title{Short injective proofs of the Erd\H{o}s-Ko-Rado and Hilton-Milner
Theorem:\\
A canonical partition of shifted intersecting set systems}
\author[Nguyen Trong Tuan]{Nguyen Trong Tuan${}^{1,2}$}
\address{${}^1$Faculty of Mathematics and Computer Science, University of Science, Ho Chi Minh City, Vietnam.} 
\address{${}^2$Vietnam National University, Ho Chi Minh City, Vietnam.}
\email{nttuan@ptnk.edu.vn (Nguyen Trong Tuan)}
\author[Nguyen Anh Thi]{Nguyen Anh Thi$^{1,2}$}
\email{nathi@hcmus.edu.vn (Nguyen Anh Thi)}

\begin{document}

\begin{abstract}
We give a canonical partition of shifted intersecting set systems, from which one can obtain unified and elementary proofs of the Erd\H{o}s-Ko-Rado and Hilton-Milner Theorem, as well as a characterization of maximal shifted $k$-uniform intersecting set systems over $[n]$. 
\end{abstract}
\maketitle

\section{Introduction}

A $k$-uniform intersecting set system on $[n]=\{1,2,\dots, n\}$ is
a collection of $k$-element subsets of $[n]$ such that any two subsets
in the collection have nonempty intersection. One of the classical
and fundamental result in extremal set theory is the Erd\H{o}s-Ko-Rado
Theorem \cite{EKR} on the maximum size of a $k$-uniform intersecting
set system. In particular, it shows that the maximum size of a $k$-uniform
intersecting set system is attained by the collection of all $k$-element
subsets of $[n]$ containing a fixed element. 
\begin{theorem}
\label{thm:EKR}Let ${\cal F}$ be a $k$-uniform intersecting set
system on $[n]$. If $n\ge2k$, then $|{\cal F}|\le\binom{n-1}{k-1}$.
Furthermore, if $n>2k$, equality occurs if and only there exists
$x$ such that ${\cal F}$ consists of all subsets of size-$k$ of
$[n]$ containing $x$. 
\end{theorem}

A strengthening of the Erd\H{o}s-Ko-Rado Theorem was shown by Hilton
and Milner \cite{HM}, which gives a tight upper bound on the size
of an intersecting $k$-uniform set system in which no element is
contained in all sets in the system. 

\begin{theorem}
\label{thm:HM}Let ${\cal F}$ be a $k$-uniform intersecting set
system on $[n]$. Assume that $n>2k$, and no element is contained
in all sets of ${\cal F}$. Then $|{\cal F}|\le\binom{n-1}{k-1}-\binom{n-k-1}{k-1}+1$. Furthermore, when $k\ge 4$, equality is attained if and only if $\cal{F}$ is isomorphic to the family $\{A\in \binom{[n]}{k}:1\in A\} \cup \{\{2,3,\dots,k+1\}\}$; and when $k=3$, equality is attained if and only if $\cal{F}$ is isomorphic to either the family $\{A\in \binom{[n]}{3}:1\in A\} \cup \{\{2,3,4\}\}$, or the family $\{A\in \binom{[n]}{3}:|A\cap \{1,2,3\}|=2$.
\end{theorem}

Since the original papers, there have been several alternative proofs
of the Erd\H{o}s-Ko-Rado and Hilton-Milner Theorem, \cite{FF,Fr19,FT,KZ,M}.
In this paper, we provide short elementary injective proofs of the
Erd\H{o}s-Ko-Rado and Hilton-Milner Theorem on intersecting set systems (with equality cases characterized).
Our approach is based on a canonical partition of the set system,
and allows to prove the Hilton-Milner Theorem with a unifying and essentially identical argument as for the Erd\H{o}s-Ko-Rado Theorem. 

The canonical partition of the set system also allows to give a characterization
of maximal shifted $k$-uniform intersecting set systems over $[n]$
via appropriate intersecting set systems over $[2k-1]$ for any $n$. As an immediate corollary, we can show that the number of maximal shifted $k$-uniform intersecting set systems over $[n]$ is bounded by a constant depending only on $k$. 

We expect that the same argument would be useful for studying intersecting
set systems with additional constraints. 

\section{Main results}

\subsection{Compression }

First, we introduce the standard compression (shifting) technique
together with its basic properties. 
\begin{defn}
A set system ${\cal F}$ on the universe $[n]=\{1,2,\dots, n\}$ is
said to be \textit{shifted} if for all $i<j$ and all $S\in{\cal F}$
with $j\in S$ and $i\notin S$, the set $S'=S-\{j\}+\{i\}\in{\cal F}$.
\end{defn}

Any set system can be transformed into a shifted set system by using
the following operations: Define an $(i,j)$-shift of ${\cal F}$
to be the set system ${\cal F}'$ obtained by replacing $S\in{\cal F}$
with $j\in S$ and $i\notin S$ with $S'=S-\{j\}+\{i\}$ if $S'\notin{\cal F}$.
It is easy to verify that upon applying an $(i,j)$-shift, an intersecting
set system ${\cal F}$ remains intersecting. Furthermore, upon finitely
many $(i,j)$-shifts, we obtain a shifted set system. 

For our later application to the Hilton-Milner Theorem, we will need
the following result of Frankl \cite{Fr19}.
\begin{lem}
Let $n\ge2k\ge4$. Suppose that ${\cal F}$ is a $k$-uniform intersecting
set system of $[n]$ with no element contained in all sets of ${\cal F}$.
Then there exists an intersecting set system ${\cal F}'$ with no
element contained in all sets of ${\cal F}'$, $|{\cal F}'|=|{\cal F}|$,
and ${\cal F}'$ is shifted. 
\end{lem}

In particular, for the proof of Theorem \ref{thm:EKR} and Theorem
\ref{thm:HM}, we can assume without loss of generality that the set
system ${\cal F}$ is shifted.

\subsection{Partitioning the set system }\label{sec:Partitioning}

In this subsection, we introduce the key idea in our elementary proofs
of Theorems \ref{thm:EKR} and \ref{thm:HM}, based on a canonical
way to partition a shifted intersecting set system. Throughout this
subsection and the remaining part of the paper, we assume that the
intersecting set system ${\cal F}$ is intersecting and shifted. 

Given two sets $A,B$ of integers of size $k$, we write $A\succeq B$
(or $B\preceq A$) if the elements of $A$ are $a_{1}<a_{2}<\dots<a_{k}$
and the elements of $B$ are $b_{1}<b_{2}<\dots<b_{k}$ and $a_{i}\ge b_{i}$
for all $i=1,\dots,k$. 
\begin{claim}
If ${\cal F}$ is shifted and $A\in{\cal F}$ then $B\in{\cal F}$
for any $B\preceq A$. 
\end{claim}

\begin{proof}
Assume that $B\preceq A$. Since $b_{1}\le a_{1}$ and ${\cal F}$
is shifted, we have $\{b_{1},a_{2},\dots,a_{k}\}\in{\cal F}$ (by
considering the $(b_{1},a_{1})$-shift of ${\cal F}$ if $b_{1}\ne a_{1}$).
Inductively we can guarantee that $\{b_{1},\dots,b_{i},a_{i+1},\dots,a_{k}\}\in{\cal F}$
for any $i\le k$, so $B\in{\cal F}$. 
\end{proof}
\begin{lem}
\label{lem:intersecting-interval}If ${\cal F}$ is intersecting and
shifted and $S\in{\cal F}$ has size $k$, then there exists $i\in[0,k-1]$
with $|S\cap[2k-i-1]|\ge k-i$. 
\end{lem}

\begin{proof}
Assume that $|S\cap[2k-i-1]|<k-i$ for all $i\in[0,k-1]$. Define
$T$ to be the set of the first $k$ integers not contained in $S$.
We claim that $T\preceq S$. Indeed, let the elements of $T$ be $t_{1}<t_{2}<\dots<t_{k}$.
If $t_{j}>s_{j}$ for some $j\in[1,k]$, then $|S\cap[t_{j}-1]|\ge j$. Furthermore, $t_j \le k+j$ since any integer less than $t_j$ is either of the form $t_i,i<j$ or an element of $S$. 
Thus, $|S\cap[2k-(k-j)-1]|\ge|S\cap[t_{j}-1]|\ge j$,
contradicting our assumption on $S$. Hence, there exists a set $T\preceq S$
with $T$ disjoint from $S$. Since ${\cal F}$ is shifted, $T\in{\cal F}$,
contradicting our assumption that ${\cal F}$ is intersecting. 
\end{proof}
We next introduce our key partition of the set system ${\cal F}$.
For each $S\in{\cal F}$, by Lemma \ref{lem:intersecting-interval},
there exists $i\in[0,k-1]$ such that $|S\cap[2k-i-1]|\ge k-i$. Let
$i_{S}$ be the smallest such $i$. We then define the sub-collection
of sets ${\cal F}_{i}$ to be the sets $S$ for which $i_{S}=i$.
We say that $S\in{\cal F}_{i}$ has \textit{type $i$}. 

\begin{lem}
\label{lem:partition}We have ${\cal F}=\bigcup_{i=0}^{k-1}{\cal F}_{i}$,
and for any $S\in{\cal F}_{i}$, $|S\cap[2k-i-1]|=k-i$ and $2k-i\notin S$. 
\end{lem}

\begin{proof}
For each $S\in{\cal F}$, let $i=i_{S}$ be the smallest integer in
$[0,k-1]$ such that $|S\cap[2k-i-1]|\ge k-i$ (so $S\in{\cal F}_{i}$).
If $i=0$, then $|S\cap[2k-1]|\ge k$ so $|S\cap[2k-1]|=k$ and $S\setminus[2k-1]=\emptyset$,
so $S\in{\cal F}_{0}$. Next, assume that $i\ge1$. Since $i$ is
smallest, $|S\cap[2k-i]|<k-i+1$ and $|S\cap[2k-i-1]|\ge k-i$. In
particular, $|S\cap[2k-i-1]|=|S\cap[2k-i]|=k-i$, which implies $2k-i\notin S$,
so $S\in{\cal F}_{i}$. 
\end{proof}
For each set $S\in{\cal F}_{i}$, define $\pi_{i}(S)=S\cap[2k-i-1]$
and $\psi_{i}(S)=S\setminus[2k-i-1]\subseteq[2k-i+1,n]$. 
\begin{lem}
\label{lem:intersecting-Fi}The collection of sets $\pi_{i}({\cal F}_{i}):=\{\pi_{i}(S),S\in{\cal F}_{i}\}$
is intersecting. 
\end{lem}

\begin{proof}
Assume that there exists $S,S'\in{\cal F}_{i}$ with $\pi_{i}(S)\cap\pi_{i}(S')=\emptyset$.
Then $|\pi_{i}(S)\cup\pi_{i}(S')|=2(k-i)$ and $|[2k-i]\setminus(\pi_{i}(S)\cup\pi_{i}(S'))|=i$.
Let $T=[2k-i]\setminus(\pi_{i}(S)\cup\pi_{i}(S'))$. It is trivial
that $T\preceq\psi_{i}(S')$, and thus, using shifts we can obtain
the set $\pi_{i}(S')\cup T$ from $S'$. In particular, $\pi_{i}(S')\cup T\in{\cal F}$,
but this contradicts the assumption ${\cal F}$ is intersecting as
$S\cap(\pi_{i}(S')\cup T)=\emptyset$. 
\end{proof}

\subsection{Proof of the Erd\H{o}s-Ko-Rado Theorem}
\begin{proof}[Proof of Theorem \ref{thm:EKR}]
We prove the theorem by induction on $k$ and on $n$. When $k=1$,
the claim is trivial for any $n\ge1$. 

Next, consider $k\ge2$. For $n=2k$, the conclusion directly follows
as subsets of $[2k]$ of size $k$ can be partitioned into $\frac{1}{2}\binom{2k}{k}=\binom{2k-1}{k-1}$
pairs of sets, such that the sets in each pair are disjoint. Thus,
each intersecting system of $k$-sets in $[2k]$ has size at most
$\binom{2k-1}{k-1}$. 

Next, consider $n>2k$. Observe that 
\[
|{\cal F}|\le\sum_{i=0}^{k-1}|{\cal F}_{i}|.
\]
For $i>0$, we have that $\pi_{i}({\cal F}_{i})$ is a $(k-i)$-uniform
intersecting set system on the universe $[2k-i-1]$, and by the inductive
hypothesis, $|\pi_{i}({\cal F}_{i})|\le\binom{2k-i-2}{k-i-1}$. Since
each set in ${\cal F}_{i}$ can be chosen by picking a set in $\pi_{i}({\cal F}_{i})$,
and picking the remaining $i$ elements in $[n]\setminus[2k-i]$ in
$\binom{n-2k+i}{i}$ ways, 
\[
|{\cal F}_{i}|\le|\pi_{i}({\cal F}_{i})|\binom{n-2k+i}{i}\le\binom{2k-i-2}{k-i-1}\binom{n-2k+i}{i}.
\]
Furthermore, consider the set system ${\cal F}'$ as follows. For
each set $S\in{\cal F}_{i}$, define a set $S'=\pi_{i}(S)\cup\{2k-i+1,\dots,2k\}$
and include it in ${\cal F}'$. Then $|{\cal F}'|=\sum_{i=0}^{k-1}|\pi_{i}({\cal F}_{i})|$.
Furthermore, ${\cal F}'$ is intersecting: since ${\cal F}$ is shifted,
$S\in{\cal F}_{i}$ implies that $S'\in{\cal F}$. Hence, ${\cal F}'$
is an intersecting system of $k$-sets of $[2k]$, which satisfies
$|{\cal F}'|\le\binom{2k-1}{k-1}$ by the base case above. Thus, $\sum_{i=0}^{k-1}|\pi_{i}({\cal F}_{i})|\le\binom{2k-1}{k-1}$.

Hence, we have 
\begin{align*}
|{\cal F}| & \le\sum_{i=1}^{k-1}|\pi_{i}({\cal F}_{i})|\binom{n-2k+i}{i}+|{\cal F}_{0}|\\
 & \le\binom{2k-1}{k-1}+\sum_{i=1}^{k-1}|\pi_{i}({\cal F}_{i})|\left(\binom{n-2k+i}{i}-1\right)\\
 & \le\binom{2k-1}{k-1}+\sum_{i=1}^{k-1}\binom{2k-i-2}{k-i-1}\left(\binom{n-2k+i}{i}-1\right)\\
 & =\sum_{i=0}^{k-1}\binom{2k-i-2}{k-i-1}\binom{n-2k+i}{i}+\binom{2k-1}{k-1}-\binom{2k-2}{k-1}-\sum_{i=1}^{k-1}\binom{2k-i-2}{k-1}\\
 & =\binom{n-1}{k-1},
\end{align*}
where we have used that 
\begin{equation}
\binom{2k-2}{k-1}+\sum_{i=1}^{k-1}\binom{2k-i-2}{k-1}=\sum_{j=0}^{k-1}\binom{k-1+j}{k-1}=\sum_{j=0}^{k-1}\left(\binom{k+j}{k}-\binom{k+j-1}{k}\right)=\binom{2k-1}{k-1},\label{eq:k-1-formula}
\end{equation}
and 
\begin{equation}
\sum_{i=0}^{k-1}\binom{2k-i-2}{k-i-1}\binom{n-2k+i}{i}=\binom{n-1}{k-1}.\label{eq:prod-formula}
\end{equation}
This latter equality can be justified as follows: for each subset
of $[n-1]$ of size $k-1$, there is a unique $i\in[0,k-1]$ so that
the set contains exactly $i$ elements in $[2k-i,n-1]$ and $k-i-1$
elements in $[2k-i-2]$ ($i$ . can be uniquely written as the union
of a subset of $[2k-i,n-1]$ of size $i$ and a subset of $[2k-i-2]$
of size $k-i-1$ (i.e., $i$ is the largest integer so that the set
contains at least $i$ elements in $[2k-i,n-1]$). 

This shows the desired inequality. Furthermore, to attain equality,
when $n>2k$, it must be the case that $|\pi_{i}({\cal F}_{i})|=\binom{2k-i-2}{k-i-1}$
for each $i\ge1$ and $|{\cal F}_{0}|=\binom{2k-2}{k-1}$. By the
inductive hypothesis, for $i\ge1$, $\pi_{i}({\cal F}_{i})$ must
only consist of sets that contain $1$ (recall that $\pi_{i}({\cal F}_{i})$
is shifted). Then ${\cal F}_{0}$ cannot contain any set without $1$
as well, since the complement of that set in $[2k]$ must be a set
of type $i\ge1$ that contains $1$. In particular, ${\cal F}$ can
only consist of sets containing $1$. 
\end{proof}

\subsection{Proof of the Hilton-Milner Theorem }
\begin{proof}[Proof of Theorem \ref{thm:HM}]
We follow a similar scheme to our proof of Theorem \ref{thm:EKR}.
Consider $n>2k$ and the same decomposition of ${\cal F}$ into the
families ${\cal F}_{i}$. We again have that 
\[
|{\cal F}|\le\sum_{i=0}^{k-1}|{\cal F}_{i}|,
\]
and as before, for $i>0$, $|\pi_{i}({\cal F}_{i})|\le\binom{2k-i-2}{k-i-1}$
and 
\[
|{\cal F}_{i}|\le|\pi_{i}({\cal F}_{i})|\binom{n-2k+i}{i}\le\binom{2k-i-2}{k-i-1}\binom{n-2k+i}{i}.
\]
We also have 
\[
\sum_{i=0}^{k-1}|\pi_{i}({\cal F}_{i})|\le\binom{2k-1}{k-1}.
\]
Since ${\cal F}$ contains at least one set that does not contain
$1$ and ${\cal F}$ is shifted, it must be the case that ${\cal F}$
contains $\{2,3,\dots,k+1\}$. Then, ${\cal F}$ cannot contain the
set $\{1,k+2,\dots,2k\}$, and in particular, ${\cal F}$ contains
no set of type $k-1$, i.e. $|{\cal F}_{k-1}|=|\pi_{k-1}({\cal F}_{k-1})|=0$. 

Hence, we have that 
\begin{align*}
|{\cal F}| & \le\sum_{i=0}^{k-1}|{\cal F}_{i}|\le\binom{2k-1}{k-1}+\sum_{i=1}^{k-2}|\pi_{i}({\cal F}_{i})|\left(\binom{n-2k+i}{i}-1\right)\\
 & \le\binom{2k-1}{k-1}+\sum_{i=1}^{k-2}\binom{2k-i-2}{k-i-1}\left(\binom{n-2k+i}{i}-1\right)\\
 & =\binom{2k-1}{k-1}+\sum_{i=1}^{k-1}\binom{2k-i-2}{k-i-1}\left(\binom{n-2k+i}{i}-1\right)-\left(\binom{n-k+1}{k-1}-1\right)\\
 & =\binom{n-1}{k-1}-\left(\binom{n-k+1}{k-1}-1\right),
\end{align*}
as desired (here we have used (\ref{eq:k-1-formula}) and (\ref{eq:prod-formula})).

Equality occurs only if $|\pi_{i}({\cal F}_{i})|=\binom{2k-i-2}{k-i-1}$
for each $i\in[1,k-2]$, and $|{\cal F}_{0}|=\binom{2k-2}{k-1}+1$.
We then have that each ${\cal F}_{i}$ with $2\le i\le k-2$ is the collection
of all sets of type $i$ which contains $1$. 

If $k\ge 4$, since one can check that the set $\{1,k+1,k+3,\dots,2k\}$ has type $k-2$, ${\cal F}_{k-2}$ contains this set. In particular, the set $\{2,3,\dots,k,k+2\}$ is not contained in $\cal{F}$. Since $\cal{F}$ is shifted, this implies that there can be no set in $\cal{F}$ not containing $1$ that is different from the set $\{2,3,\dots,k+1\}$. Hence, we conclude
that if equality occurs, then ${\cal F}$ must be the family consisting
of sets containing $1$ that intersect $\{2,3,\dots,k+1\}$, together
with the set $\{2,3,\dots,k+1\}$. 

If $k=3$, then $\pi_1(\cal{F}_1)$ is a shifted intersecting family of size $3$ consisting of subsets of size $2$ in $\{1,2,3,4\}$. One can then check that $\pi_1(\cal{F}_1)$ is either the family of sets containing the element $1$, or the family \linebreak $\{\{1,2\},\{1,3\},\{2,3\}\}$. From this, one can obtain that $\cal{F}$ is either the family of sets containing the element $1$ together with $\{2,3,4\}$, or the family of sets intersecting $\{1,2,3\}$ in a subset of size exactly $2$. 
\end{proof}
\subsection{A characterization of maximal shifted intersecting families}

Here we record a characterization of maximal shifted intersecting
families based on the partition of the set system in Section \ref{sec:Partitioning}.
Let ${\cal F}$ be a maximal intersecting family which is shifted.
Given a subset $A'$ of $[2k-i]$ of size $k-i$, we denote by ${\cal S}_{i}(A')=\{B\subseteq[n]:\pi_{i}(B)=A',B\setminus\pi_{i}(B)\subseteq[2k-i+1]\}$.
For each $A$ of type $i$, define ${\cal S}(A)={\cal S}_{i}(\pi_{i}(A))$. 
\begin{lem}
\label{lem:characterization}Let $A\in{\cal F}$ be set of type $i$.
Then ${\cal F}$ contains ${\cal S}(A)$.
\end{lem}

\begin{proof}
We show that any set $B\in{\cal S}(A)$ intersects all sets in ${\cal F}$.
Indeed, assume that $C\in{\cal F}$ is disjoint from $B$. Then $C\cap\pi_{i}(A)=\emptyset$.
Since ${\cal F}$ is shifted, by shifting $C$, we obtain that $[2k-i]\setminus\pi_{i}(A)\in{\cal F}$
(note that $|[2k-i]\setminus\pi_{i}(A)|=k$), which is a contradiction
as $A\cap([2k-i]\setminus\pi_{i}(A))=\emptyset$. 

Since ${\cal F}$ is maximal, we then have that ${\cal S}(A)\subseteq{\cal F}$. 
\end{proof}
\begin{corollary}
There is a bijection between maximal shifted intersecting families
on $[n]$ and shifted intersecting set systems ${\cal G}=\bigcup_{i=0}^{k-1}{\cal G}_{i}$
over $[2k-1]$ where ${\cal G}_{i}$ consists of sets of size $k-i$
contained in $[2k-i-1]$. 
\end{corollary}

\begin{proof}
To each maximal shifted intersecting family ${\cal F}$ on $[n]$,
we can associate a shifted intersecting set system ${\cal G}=\bigcup_{i=0}^{k-1}{\cal G}_{i}$
over $[2k-1]$ where ${\cal G}_{i}$ consists of sets of size $k-i$
contained in $[2k-i-1]$, by defining ${\cal G}_{i}$ to be the collection
of $\pi_{i}(A)$ for $A\in{\cal F}$ of type $i$. ${\cal G}$ is
an intersecting set system by Lemma \ref{lem:intersecting-Fi}. Conversely,
for each shifted intersecting set system ${\cal G}=\bigcup_{i=0}^{k-1}{\cal G}_{i}$
over $[2k-1]$, by Lemma \ref{lem:characterization}, we can recover
the maximal shifted intersecting set system ${\cal F}$ on $[n]$
as the union of ${\cal S}_{i}(A_{i})$ for $A_{i}\in{\cal G}$ of
size $i$. One can easily check that this gives a bijection between
maximal shifted intersecting families on $[n]$ and shifted intersecting
set systems ${\cal G}=\bigcup_{i=0}^{k-1}{\cal G}_{i}$ over $[2k-1]$. 
\end{proof}

\end{document}